\documentclass[12pt]{article} 

\usepackage{amsmath}
\usepackage{amsthm}
\usepackage{amsfonts}
\usepackage{mathrsfs}
\usepackage{stmaryrd}
\usepackage{setspace}
\usepackage{fullpage}
\usepackage{amssymb}
\usepackage{breqn}
\usepackage{enumitem}
\usepackage{bbold} 
\usepackage{authblk}
\usepackage{comment}
\usepackage{hyperref}
\usepackage{pgf,tikz}
\usepackage{graphicx}
\usetikzlibrary{decorations.pathreplacing,arrows}
\tikzstyle{vertex}=[circle,draw=black,fill=black,inner sep=0,minimum size=3pt,text=white,font=\footnotesize]

\newtheorem{thm}{Theorem}[section]

\newtheorem{lemma}[thm]{Lemma}

\newtheorem{proposition}[thm]{Proposition}

\newtheorem{cor}[thm]{Corollary}

\newtheorem{clm}[thm]{Claim}

\newtheorem*{lemma*}{Lemma}
\newtheorem*{proposition*}{Proposition}
\newtheorem*{theorem*}{Theorem}

\newcommand\ex{\ensuremath{\mathrm{ex}}}

\newcommand\cG{{\mathcal G}}

\newcommand\cN{{\mathcal N}}

\newcommand{\ignore}[1]{}

\title{On Tur\'an-good graphs}
\author{D\'aniel Gerbner\footnote{Alfr\'ed R\'enyi Institute of Mathematics, E-mail: \texttt{gerbner.daniel@renyi.hu.} Research supported by the
    National Research, Development and Innovation Office -- NKFIH under the
    grants FK 132060, KKP-133819, KH130371 and SNN 129364.}}
    
    \date{}

\begin{document}

\maketitle

\begin{abstract}
     For graphs $H$ and $F$, the generalized Tur\'an number $\ex(n,H,F)$ is the largest number of copies of $H$ in an $F$-free graph on $n$ vertices. We say that $H$ is $F$-Tur\'an-good if $\ex(n,H,F)$ is the number of copies in the $(\chi(F)-1)$-partite Tur\'an graph, provided $n$ is large enough.

We present a general theorem in case $F$ has an edge whose deletion decreases the chromatic number. In particular, this determines $\ex(n,P_k,C_{2\ell+1})$ and $\ex(n,C_{2k},C_{2\ell+1})$ exactly, if $n$ is large enough. We also study the case when $F$ has a vertex whose deletion decreases the chromatic number.
\end{abstract}

\section{Introduction}

A basic result in extremal Combinatorics is Tur\'an's theoem \cite{T}. It states that a $K_k$-free graph on $n$ vertices cannot have more edges than the Tur\'an graph $T_{k-1}(n)$, which is the complete $(k-1)$-partite graph where each partite class has cardinality $\lfloor n/(k-1)\rfloor$ or $\lceil n/(k-1)\rceil$. In general, Tur\'an theory deals with the function $\ex(n,F)$, which is the largest number of edges in $n$-vertex $F$-free graphs. Let $\cN(H,G)$ denote the number of copies of $H$ in $G$. Generalized Tur\'an theory deals with $\ex(n,H,F):=\max\{\cN(H,G): G \text{ is an $n$-vertex $F$-free graph}\}$, i.e. the largest number of copies of $H$ in $F$-free graphs on $n$ vertices. 

After several sporadic results (see e.g. \cite{BGY2008,G2012,gypl,HaETAL,LiGy,zykov}), the systematic study of this problem was initiated by Alon and Shikhelman \cite{as}. 
Since then, this problem has attracted several researchers, see e.g. \cite{cc,chase,gerbner2,GGMV2017+,gp2,gp3,gs2017,gstz,mq}. 

However, there are not many exact results in this area (by exact result we mean that for given $H$ and $F$, we know the value of $\ex(n,H,F)$ for every $n$ large enough). 
Most of the exact results are when the Tur\'an graph contains the most copies of $H$. Gy\H ori, Pach and Simonovits \cite{gypl} examined for what graphs $H$ do we have that $\ex(n,H,K_{k+1})=\cN(H,T_k(n))$. Gerbner and Palmer \cite{gp3} extended these investigations for arbitrary $k$-chromatic graphs.
Following them, given a graph $F$ with $\chi(F)=k$, we say that $H$ is \textit{$F$-Tur\'an-good} if $\ex(n,H,F)=\cN(H,T_{k-1}(n))$ and $H$ does not contain $F$. If $F=K_k$, we use the briefer term \textit{$k$-Tur\'an good}. Let us state the main result of Gy\H ori, Pach and Simonovits \cite{gypl} using this term.

\begin{thm}[Gy\H ori, Pach and Simonovits \cite{gypl}]\label{gpl} 
Let $r\ge 3$ and let $H$ be a $(k-1)$-partite graph with $m>k-1$ vertices,
containing $\lfloor m/(k-1)\rfloor$ vertex disjoint copies of $K_{k-1}$. Suppose
further that for any two vertices $u$ and $v$ in the same connected component of $H$, there is a sequence $A_1,\dots,A_s$ of $(k-1)$-cliques in $H$ such that $u\in A_1$, $v\in A_s$, and for any $i<s$, $A_i$ and $A_{i+1}$ share $k-2$ vertices.
Then $H$ is $k$-Tur\'an-good. Moreover, if $n$ is large enough, the Tur\'an graph is the only $K_k$-free graph with $\ex(n,H,K_k)$ copies of $H$.
\end{thm}

Gerbner and Palmer \cite{gp3} obtained a theorem of a similar flavor.

\begin{thm}[Gerbner, Palmer \cite{gp3}]\label{turgood}
	Let $H$ be a $k$-Tur\'an-good graph. Let $H'$ be any graph constructed from $H$ in the following way.
	Choose a complete subgraph of $H$ with vertex set $X$, add a vertex-disjoint copy of $K_{k-1}$ to $H$ and join the vertices in $X$ to the vertices of $K_{k-1}$ by edges arbitrarily.
	Then $H'$ is $k$-Tur\'an-good.
\end{thm} 

Neither of the above two theorems imply the other. The main difference is that vertices in the additional clique can be connected to anything in Theorem \ref{gpl}, but only to vertices of another clique in Theorem \ref{turgood}. The trade-off is the necessity of the strong connection property of the cliques in Theorem \ref{gpl}. Let us remark that for $k=3$, the assumptions of Theorem \ref{gpl} are nothing else but that $H$ is bipartite and has a matching of size $\lfloor |V(H)|/2\rfloor$; the property of the sequence of 2-cliques reduces to the property that every connected component is connected.

Observe that on their own, both theorems use $K_{k-1}$ as building blocks, and can be used only for graphs mostly covered by vertex-disjoint copies of $K_{k-1}$.
Therefore, another difference is that in Theorem \ref{turgood} we can start with an arbitrary $k$-Tur\'an-good graph, and add cliques afterwards. Here we prove such a strengthening for Theorem \ref{gpl}.

\begin{proposition}\label{newturgoo}
 Let $H$ be a $k$-Tur\'an-good graph with a unique proper $(k-1)$-coloring. Let $H'$ consist of $H$ and a copy $K$ of $K_{k-1}$ with vertices $v_1,\dots,v_{k-1}$, with additional edges between $V(H)$ and $V(K)$ such that for every $i\le k-1$, there is a copy of $K_{k-1}$ in $H'$ containing $v_i$, but not containing any $v_j$ for $j>i$. If $H'$ has chromatic number $k-1$, then $H'$ is $k$-Tur\'an-good.
\end{proposition}

Let us show an example where this proposition is stronger then the above theorems. We will start with a slightly unbalanced complete bipartite graph. Ma and Qiu \cite{mq} showed that $K_{s,t}$ with $s\le t$ is 3-Tur\'an-good if and only if $t<s+1/2+\sqrt{2s+1/4}$. Proposition \ref{newturgoo} implies that if the vertices of a connected bipartite graph $H$ can be vertex-disjointly covered by one such $K_{s,t}$ and a matching, then $H$ is 3-Tur\'an-good.

Let us turn our attention to $F$-Tur\'an-good graphs where $F$ is not a clique. We show a weak version of the above results for this case. 
We say that an edge of a graph $G$ is a color-critical edge if deleting it from $G$ decreases its chromatic number. An $m$-chromatic graph $F$ with a color-critical edge often behaves similarly to $K_m$ in extremal problems. In particular, Simonovits \cite{miki} showed that for $n$ large enough, the Tur\'an graph $T_{m-1}(n)$ contains the most number of edges, and it was extended by Ma and Qiu \cite{mq}, who showed that $T_{m-1}(n)$ also contains the most number of $K_r$ for $r<m$. Gerbner \cite{gerbner3} proved a stability version.

\begin{lemma}[Gerbner \cite{gerbner3}]\label{stabi} Let $F$ be a $k$-chromatic graph with a color-critical edge and $r<k$. If $G$ is an $n$-vertex $F$-free graph with chromatic number more than $k-1$, then $ex(n,K_r,F)-N(G,K_r)=\Omega(n^{r-1})$.
\end{lemma}

Using this, we can extend the above theorems from $K_k$ to certain graphs with color-critical edges. The main idea is that if an $F$-free graph does not have too many copies of $K_k$, then those create only a negligible amount of copies of $H$.

\begin{thm}\label{turgoodcrit}
	Let $F$ be a $k$-chromatic graph with a color-critical edge such that $\ex(n,K_k,F)=o(n^{k-1})$, and $H$ be a graph that is both $F$-Tur\'an-good and $k$-Tur\'an-good. Let $H'$ be any graph constructed from $H$ in the following way.
		Choose a complete subgraph of $H$ with vertex set $X$, add a vertex-disjoint copy of $K_{k-1}$ to $H$ and join the vertices in $X$ to the vertices of $K_{k-1}$ by edges arbitrarily.
	Then $H'$ is $F$-Tur\'an-good. Moreover, if $G$ is an $n$-vertex $F$-free graph with chromatic number more than $m-1$, then $\ex(n,H',F)-N(H',G)=\Omega(n^{|V(H')|-1})$.
\end{thm} 

\begin{proposition}\label{newturgoocrit}	Let $F$ be a $k$-chromatic graph with a color-critical edge such that $\ex(n,K_k,F)=o(n^{k-1})$.
 Let $H$ be a graph that is both $F$-Tur\'an-good and $k$-Tur\'an-good with a unique proper $(k-1)$-coloring. Let $H'$ consist of $H$ and a copy $K$ of $K_{k-1}$ with vertices $v_1,\dots,v_{k-1}$, with additional edges between $V(H)$ and $V(K)$ such that for every $i\le k-1$, there is a copy of $K_{k-1}$ in $H'$ containing $v_i$, but not containing $v_j$ for $j>i$. If $H'$ has chromatic number $k-1$, then $H'$ is $F$-Tur\'an-good. Moreover, if $G$ is an $n$-vertex $F$-free graph with chromatic number more than $m-1$, then $\ex(n,H',F)-N(H',G)=\Omega(n^{|V(H')|-1})$.
\end{proposition}

Let us note the extra assumption $\ex(n,K_k,F)=o(n^{k-1})$. For $k=3$, an example is $F=C_{2\ell+1}$, as $\ex(n,K_3,C_{2\ell+1})=O(n^{1+1/\ell})$ due to Gy\H ori and Li \cite{LiGy}. Another example is the book $B_t$, which consists of an edge $uv$ and $t$ other vertices, that are adjacent to both $u$ and $v$. Alon and Shikhelman \cite{as} showed $\ex(n,K_3,B_t)=o(n^2)$.

Gerbner and Palmer \cite{gp3} conjectured that paths $P_m$ and even cycles $C_{2m}$ are $C_{2\ell+1}$-Tur\'an-good for any $m$ and $\ell$. They proved this conjecture for $P_4$ and $\ell=2$. They also showed that if $P_{2m}$ is $C_{2\ell+1}$-Tur\'an-good, then $C_{2m}$ is $C_{2\ell+1}$-Tur\'an-good too. Thus we can fully resolve their conjecture using Theorem \ref{turgoodcrit} or Proposition \ref{newturgoocrit}.

\begin{cor}
For any positive integers $m$ and $\ell$, $P_m$ and $C_{2m}$ are $C_{2\ell+1}$-Tur\'an-good.
\end{cor}

Gerbner and Palmer \cite{gp3} also showed that $P_4$ is $B_2$-Tur\'an-good. We can generalize it as follows.

\begin{cor} For any positive integers $m$ and $t$, 
$P_m$ is $B_t$-Tur\'an-good.
\end{cor}

Gerbner and Palmer \cite{gp3} showed an example where a graph $H$ is $F$-Tur\'an-good, and $F$ does not have a color-critical edge ($H=C_4$ and $F$ is the \textit{2-fan}, two triangles sharing a vertex). Observe that one can add additional edges to the Tur\'an graph without violating the $F$-free property in this case. However, those edges cannot create additional copies of $H$. A natural question is  for what graphs $F$ can we find an $H$ such that $H$ is $F$-Tur\'an-good?

We say that a vertex $v$ of a graph $G$ is \textit{color-critical} if by deleting $v$ from $G$ we obtain a graph with smaller chromatic number. 

\begin{thm}\label{critver}
    There exists an $F$-Tur\'an-good graph if and only if $F$ has a color-critical vertex.
\end{thm}

We prove Proposition \ref{newturgoo}, Theorem \ref{turgoodcrit} and Proposition \ref{newturgoocrit} in Section 2, and we prove Theorem \ref{critver} in Section 3. We finish the paper with some concluding remarks in Section 4.

\section{Graphs with a color-critical edge}

We start with the proof of Proposition \ref{newturgoo}, that we restate here for convenience.

\begin{proposition*}
 Let $H$ be a $k$-Tur\'an-good graph with a unique proper $(k-1)$-coloring. Let $H'$ consist of $H$ and a copy $K$ of $K_{k-1}$ with vertices $v_1,\dots,v_{k-1}$, with additional edges between $V(H)$ and $V(K)$ such that for every $i\le k-1$, there is a copy of $K_{k-1}$ in $H'$ containing $v_i$, but not containing $v_j$ for $j>i$. If $H'$ has chromatic number $k-1$, then $H'$ is $k$-Tur\'an-good.
\end{proposition*}

\begin{proof} Let $G$ be a $K_k$-free graph on $n$ vertices.
We will count the copies of $H'$ in $G$ the following way. First we pick a copy $K$ of $K_{k-1}$, then a vertex-disjoint copy $H'$ of $H$. Then we pick an actual embedding of $H$ into $H'$, and afterwards an actual embedding of $K$ into $K'$ such that the images of the remaining edges of $H'$ are present in $G$. We claim that $T_{k-1}(n)$ gives the maximum for each of the above four factors, finishing the proof.

We have $\cN(K_{k-1},G)\le \cN(H,T_{k-1}(n))$ by a theorem of Zykov \cite{zykov} (this particular case also follows from Theorem \ref{gpl}), thus the number of ways to pick a copy $K'$ of $K_{k-1}$ in $G$ is the largest if $G$ is the Tur\'an graph. Then there are at most $\ex(n-k+1,H,K_k)$ ways to pick a vertex-disjoint copy of $H$, which is at most $T_{k-1}(n-k+1)$ for $n$ large enough, as $H$ is $k$-Tur\'an-good. Observe that we have equality here in case $G=T_{k-1}(n)$, as removing a maximal clique from the Tur\'an graph gives a smaller Tur\'an graph.
The number of ways $H$ can be embedded into $H'$ is the number of isomorphisms of $H$ and does not depend on $G$. 

After $H$ is embedded, we claim that there is at most one way to finish the embedding. We pick the images of the vertices $v_i$ from $K$ one by one, in the order of their indices. For each $v_i$, it is contained in a copy $K''$ of $K_{k-1}$ that is already embedded. The other $k-2$ vertices of $K''$ are already embedded, and their images have at most one common neighbor in $K'$, as $G$ is $K_k$-free. This means we only have one vertex that can be picked as $v_i$.

Finally, we show that in the Tur\'an graph, there is a way to finish the embedding. The other $k-2$ vertices of $K''$ that are already embedded must belong to different partite classes of the Tur\'an graph, thus $v_i$ must be from the remaining partite class. We have to show that this way the $v_i$'s are mapped to different vertices. This is where we use the unique coloring property of $H$. Let $j$ be a color class in this unique coloring, then all the vertices of color $j$ are mapped to the same partite class $A_j$ of the Tur\'an graph. If $v_i$ is of color $j$ in $H'$, then the other $k-2$ vertices of $K''$ are not of color $j$, thus they are not in $A_j$, hence $v_i$ is mapped to the vertex of $K'$ that belong to $A_j$. As the vertices of $K$ belong to different color classes in $H$, they are mapped to different partite classes of the Tur\'an graph, finishing the proof.
\end{proof}

Let us continue with the proof of Theorem \ref{turgoodcrit}, that we restate here for convenience.

\begin{theorem*}	Let $F$ be a $k$-chromatic graph with a color-critical edge such that $\ex(n,K_k,F)=o(n^{k-1})$, and $H$ be a graph that is both $F$-Tur\'an-good and $k$-Tur\'an-good. Let $H'$ be any graph constructed from $H$ in the following way.
		Choose a complete subgraph of $H$ with vertex set $X$, add a vertex-disjoint copy of $K_{k-1}$ to $H$ and join the vertices in $X$ to the vertices of $K_{k-1}$ by edges arbitrarily.
	Then $H'$ is $F$-Tur\'an-good. Moreover, if $G$ is an $n$-vertex $F$-free graph with chromatic number more than $m-1$, then $\ex(n,H',F)-N(H',G)=\Omega(n^{|V(H')|-1})$.
\end{theorem*}

\begin{proof}
By a result of Ma and Qiu \cite{mq}, if $n$ is large enough, then the maximum number of copies of $K_{k-1}$ in an $F$-free graph is achieved by the Tur\'an graph $T_{k-1}(n)$. Since $H$ is $F$-Tur\'an-good, the Tur\'an graph $T_{k-1}(n-k+1)$ has the maximum number of copies of $H$ among $F$-free graphs on $n-k+1$ vertices. We will show that $T_{k-1}(n)$ has the maximum number of copies of $H'$.
	
	Let $G$ be an $F$-free graph on $n$ vertices with the maximum number of copies of $H'$. If $G$ has chromatic number at most $k-1$, then $G$ is $K_k$-free, thus we are done by Theorem \ref{turgood}. If $G$ has chromatic number more than $k-1$, then by Lemma \ref{stabi} we have  $\cN(K_{k-1},T_{k-1}(n))-\cN(K_{k-1},G)=\Omega(n^{k-2})$. 
	
We follow the proof of Theorem \ref{turgood} from \cite{gp3}. We take a complete subgraph $Y$ of $G$, disjoint from $K$, and consider the bipartite subgraph $G'$ of $G$ consisting of the edges between $K$ and $Y$. It was shown in \cite{gp3} that if $G$ is $K_k$-free, then a matching covering $Y$ is missing from $G'$. It is easy to see that if such a matching is not missing, then not only there is a $K_k$ in $G$, but there is a $K_k$ with vertices in $Y\cup V(K)$. For sake of completeness, we repeat the argument here. By Hall's theorem, there is a subset $Y'\subset Y$ such that all the vertices of $Y'$ are connected to all but less than $Y'$ vertices of $K$. Then $Y'$ and those vertices form a clique of size at least $k$.

Let us count first the copies of $H'$ such that there is no $K_k$ in $G$ on the vertex sets of them. For those copies there is a matching missing from $G'$. Observe that in the Tur\'an graph between a clique of size $k-1$ and a clique of size $|Y|$, only a matching covering the smaller clique is missing. This implies that after picking a copy of $H$, there are at least as many ways to connect the appropriate subclique of it to $K$ in the Tur\'an graph, as in $G$.

The number of such copies of $H'$ is at most the product of the number of copies of $K_{k-1}$, the number of copies of $H$ on the remaining $n-k+1$ vertices and the number of ways to join the vertices of $K_{k-1}$ and $H$, all divided by the number of times a copy of $H'$ was counted. The first quantity is less than $\cN(K_{k-1},T_{k-1}(n))$ by $\Omega(n^{k-2})$.
The second quantity is maximized by the Tur\'an graph and is $O(n^{|V(H)|})$, the third quantity is also maximized by the Tur\'an graph, while the last quantity depends only on $H'$. This implies that the number of such copies of $H'$ is $\cN(H',T_{k-1}(n))-\Omega(n^{|V(H')|-1})$.

Let us continue with the copies of $H'$ that contain a vertex set of $K_k$ in $G$. As $G$ is $F$-free, there are $o(n^{k-1})$ copies of $K_k$ in $G$, thus $o(n^{|V(H')|-1})$ copies of $H'$. Adding up the two bounds finishes the proof.
\end{proof}

Let us continue with Proposition \ref{newturgoocrit} that we restate here for convenience,
We only give a sketch of the proof, as it can be easily obtained by combining the above two proofs. We assume familiarity with those proofs.

\begin{proposition*}	Let $F$ be a $k$-chromatic graph with a color-critical edge such that $\ex(n,K_k,F)=o(n^{k-1})$.
 Let $H$ be an $F$-Tur\'an-good graph with a unique proper $(k-1)$-coloring. Let $H'$ consist of $H$ and a copy $K$ of $K_{k-1}$ with vertices $v_1,\dots,v_{k-1}$, with additional edges between $V(H)$ and $V(K)$ such that for every $i\le k-1$, there is a copy of $K_{k-1}$ in $H'$ containing $v_i$, but not containing $v_j$ for $j>i$. If $H'$ has chromatic number $k-1$, then $H'$ is $F$-Tur\'an-good. Moreover, if $G$ is an $n$-vertex $F$-free graph with chromatic number more than $m-1$, then $\ex(n,H',F)-N(H',G)=\Omega(n^{|V(H')|-1})$.
\end{proposition*}

\begin{proof}[Sketch of proof] Let $G$ be an $F$-free graph on $n$ vertices. If $G$ has chromatic number $k-1$, then it is $K_k$-free and we are done. If $G$ has chromatic number at least $k$, then by Lemma \ref{stabi} $G$ has less copies of $K_{k-1}$ than the Tur\'an graph by $\Omega(n^{k-2})$.

First we count copies of $H'$ in $G$ such that there are no $k$ vertices in that copy that induce a clique in $G$. For them, we can follow the proof of Proposition \ref{newturgoo}, with one exception: the number of $(k-1)$-cliques is less by $\Omega(n^{k-2})$, which implies that the number of copies of $H$ is less by $\Omega(n^{|V(H')|-1})$.

Then we count the other copies of $H'$: as there are $o(n^{k-1})$ copies of $K_k$ in $G$, we have that there are $o(n^{|V(H')|-1})$ such copies of $H$. Adding up the two bounds finishes the proof.
\end{proof}

\section{Graphs with a color-critical vertex}

We will use progressive induction. This is a version of induction that can be used to prove combinatorial statements that hold only for $n$ large enough. It was introduced by Simonovits \cite{miki}. It was used for a generalized Tur\'an problem in \cite{gerbner2}. The statement in \cite{miki} is very general, here we state a version adapted for generalized Tur\'an problems.

For a graph $G$ and a subgraph $G'$, we denote by $\cN_I(H,G,G')$ the number of copies of $H$ in $G$ that contain at least one vertex from $G'$. Given $H$ and $F$, we say that $G$ is an extremal graph if $\ex(n,H,F)=\cN(H,G)$.

\begin{lemma}\label{progin} Let $F$ and $H$ be graphs and $G_n$ be an $F$-free graph for every $n$.
Let $\cG_n$ be a family of $n$-vertex $F$-free graphs for every $n$, such that if $G_1,G_2\in\cG_n$, then $\cN(H,G_1)=\cN(H,G_2)$. 
Assume that there is an $n_0$ such that for every $n\ge n_0$, for every extremal graph $G$ on $n$ vertices, there is a subgraph $H'$ of $G$ with $|V(H')|\le n/2$, such that $H'$ is also the subgraph of some $G_n\in \cG_n$ and we have the following:
$\cN_I(H,G,H')\le \cN_I(H,G_n,H')$, with equality only if $G\in\cG_n$.

Then for $n$ large enough, $\ex(n,H,F)=\cN(H,G_n)$ for some $G_n\in\cG_n$. Moreover, every extremal graph belongs to $\cG_n$.
\end{lemma}

We omit the proof of this specialized version.
It follows from the original version \cite{miki} in a straightforward way, but it is also easy to see why it holds without knowing the original proof. For small $n$ it is possible that some graph has more copies of $H$ than any $G_n\in \cG_n$. However, this means a surplus of constant many copies of $H$, and then this surplus starts decreasing when $n\ge n_0$, and eventually vanishes.

\smallskip

We will also use the following result. 
\begin{thm}[Gerbner and Palmer \cite{gp2}]\label{gepa}
Let $H$ be a graph and $F$ be a graph with chromatic number $k$, then $\ex(n,H,F)\le \ex(n,H,K_k)+o(n^{|V(H)|})$.
\end{thm}

The harder part of Theorem \ref{critver} follows from the next theorem.

\begin{thm}\label{compl} Let $H$ be a complete $k$-partite graph $K_{b,\dots,b}$ and $F$ be a complete $(k+1)$-partite graph $K_{1,a,\dots,a}$ such that $b>2a-2$. Then $H$ is $F$-Tur\'an-good. Moreover, every extremal graph contains $T_k(n)$.
\end{thm}

\begin{proof}
Let $\cG_n$ be the family of $F$-free graphs containing $T_k(n)$, and observe that they each contain the same number of copies of $H$. Indeed, fix a $G_n\in\cG_n$ and a copy of $T_k(n)$ in it, and let us call the edges not in that copy \textit{extra edges}. Then a vertex $v$ can be incident to at most $a-1$ extra edges. If a copy of $H$ contains an extra edge between $u$ and $v$, that means each other vertex of that $H$ is adjacent to at least one of $u$ and $v$, thus $H$ contains at most $2a-4$ other vertices from that partite set of $T_k(n)$ where $u$ and $v$ belong. Therefore, if a copy of $H$ in $G_n$ contains a set $U$ of more than $2a-2$ vertices from a partite set of $T_k(n)$, then there are no extra edges inside $U$, thus no edges at all inside $U$. Thus $U$ is a subset of a partite set of $H$, in particular $|U|\le b$. Thus the only way to choose $kb$ vertices from the $k$ partite sets of $T_k(n)$ is to choose $b$ from each, and they have to form the partite sets of $H$, thus $H$ does not use any extra edge.

\smallskip

Let $G$ be an $F$-free graph on $n$ vertices with $\cN(H,G)=\ex(n,H,F)$ and assume indirectly that progressive induction (Lemma \ref{progin}) cannot be applied to finish the proof, i.e. for any subgraph $H'$ of $G$ with $|V(H')|\le n/2$, we do not have that $\cN_I(H,G,H')\le \cN_I(H,G,T_k(n))$ with equality only if $G\in \cG_n$. Let $x$ denote the minimum number of copies of $H$ that a vertex in the Tur\'an graph $T_k(n)$ is contained in.

\begin{clm}\label{turi} Every vertex $v$ is contained in at least $x$ copies of $H$ in $G$.
\end{clm}

\begin{proof}
Otherwise let $H'$ be the graph containing only $v$, and we can apply Lemma \ref{progin} to finish the proof, a contradiction.
\end{proof}
Observe that if the degree of $v$ is $o(n)$, then we are in this situation: every copy of $H$ that contains $v$ also contains at least one vertex in its neighborhood, thus we can count the copies of $H$ containing $v$ by picking a neighbor of $v$ ($o(n)$ ways), then picking $|V(H)|-2$ other vertices ($O(n^{|V(H)|-2})$ ways).

\smallskip

Therefore, we can assume that every vertex has linear degree. We will try to apply Lemma \ref{progin} with $H'$ being the complete $k$-partite graph $K_{c,\dots,c}$ with $c=2a-1$. If $H'$ is not a subgraph of $G$, then $\cN(H,G)=0$. 
We pick a copy of $H'$ with partite sets $A_1,\dots, A_k$, and with a slight abuse of notation we denote that copy by $H'$ and its vertex set by $U'$, and we also denote a copy of $H'$ in $T_k(n)$ by $H'$. 

To apply Lemma \ref{progin} in this case, we need to show that $\cN_I(H,G,H')\le \cN_I(H,T_k(n),H')$, with equality only if $G\in\cG_n$. Assume that $\cN_I(H,G,H')\ge \cN_I(H,T_k(n),H')$. Observe that the number of copies of $H$ containing $t$ vertices from $U'$ is $\Theta(n^{|V(H)|-t})$ in $T_k(n)$, and $O(n^{|V(H)|-t})$ in $G$. Therefore, we will focus on the main term $n^{|V(H)|-1}$, thus on the copies of $H$ that contain exactly one vertex from $U'$. Let $H_0$ denote the complete $k$-partite graph $K_{b-1,b,b,\dots,b}$. 

Let $G'$ be the subgraph of $G$ induced on the vertices not in $U'$. Observe that if a vertex in $G'$ is connected to each $A_i$ by at least $a$ vertices, then they form a copy of $F$, a contradiction. Thus for each vertex $v$ of $G'$, there is at least one $A_i$ such that $v$ is connected to at most $a-1$ vertices of $A_i$. Note that $v$ can be connected to each vertex of some other $A_j$. We say that a copy of $H_0$ in $G'$ is \textit{nice} if all the vertices of the $k-1$ larger parts of this copy are connected to all the $c$ vertices in some $A_j$.

\begin{clm}\label{clam} All but
$o(n^{|V(H_0)|})$ copies of $H_0$ in $G'$ are nice.
\end{clm}

\begin{proof}
Let us consider a copy of $H$ that contains exactly one vertex from $U'$. It means it contains a copy of the complete $k$-partite graph $H_0:=K_{b-1,b,b,\dots,b}$ in $G'$. By Theorem \ref{gepa}, $G'$ contains at most $(1+o(1))\ex(n-kc,H_0,K_{k+1})$ copies of $H_0$. Theorem \ref{gpl} shows that $T_k(m)$ is $(k+1)$-Tur\'an-good for every $m$, thus we have that $\ex(n-kc,H_0,K_{k+1})=\cN(H_0,T_k(n-kc))$.

It is easy to see that in the Tur\'an graph, every copy of $H_0$ that avoids the selected copy of $H'$ can be extended to a copy of $H$ with one of $c$ vertices of that copy of $H'$ (those in the same partite class of the Tur\'an graph). We claim that in $G$, every copy of $H_0$ can be extended to a copy of $H$ with at most $c$ of the vertices of $H'$. Indeed, let $U''$ be the set of vertices in $U'$ that are connected to every vertex in the $(k-1)$ larger partite sets of $H_0$. If $U''$ intersects some partite sets in at least $a$ vertices, and has another vertex $v$ in another partite set, then $v$, those $a$ vertices, and $a$ vertices from each of the the $k-1$ larger partite sets of $H_0$ form a copy of $F$, a contradiction.
 If $|U''|> c$, then this is the case.
Moreover, if $|U''|=c$, then we have the same situation, unless $U''$ is a partite set of $H'$.  

The main term of $\cN_I(H,G,H')$ is at most $(1+o(1))\cN(H_0,T_k(n-kc))$ times $c$, minus the number of those 
copies of $H_0$ in $G'$ that are not connected to the $c$ vertices of one part of $H'$ (as they should be counted at most $c-1$ times). If the last term is $\Omega(n^{|V(H_0)|})$, then the main term is smaller than the main term of $\cN_I(H,T_k(n),H')$, a contradiction finishing the proof of the claim.
\end{proof}

Let us return to the proof of the theorem. Consider a copy of $H$ in $G'$. We say that a vertex $v$ of it is \textit{replaceable} (with respect to that copy) if deleting $v$ we obtain a nice copy of $H_0$, ie. all the vertices of the $k-1$ larger parts of that copy of $H_0$ are connected to the $c$ vertices of one part of $H'$. 
If that part is $A_i$, we say that $v$ is replaceable by $A_i$.

Consider an arbitrary vertex $v$ in $G'$. We know that $v$ is in $\Omega(n^{|V(H)|-1})$ copies of $H$ in $G$, and $\Omega(n^{|V(H)|-2})$ of those share a vertex with $H'$. This implies that there are $\Omega(n^{|V(H)|-1})$ copies of $H_0$ in $G'$ that can be extended to a copy of $H$ with $v$. $\Omega(n^{|V(H)|-1})$ of those copies of $H_0$ can also be extended to a copy of $H$ with any one vertex from $A_i$ for some $i\le k$ by Claim \ref{clam}, i.e. $v$ is replaceable with respect to $\Omega(n^{|V(H)|-1})$ copies of $H$. 
Let $B_i$ denote the set of vertices in $G'$ such that there are $\Omega(n^{|V(H)|-1})$ copies of $H_0$ in $G'$ that can be extended to a copy of $H$ with $v$ or with any vertex of $A_i$, i.e. $v$ is replaceable by $A_i$ with respect to $\Omega(n^{|V(H)|-1})$ copies of $H$. 

Observe that every $v\in B_i$ is connected to less than $a$ vertices of $A_i$. Indeed, let us consider a copy of $H_0$ that $v$ extends to a copy of $H$. If $v$ is connected to $a$ vertices of $A_i$, then we can take $a$ vertices from each of the $k-1$ larger partite classes of $H_0$, $a$ neighbors of $v$ from $A_i$ and $v$ to obtain a copy of $F$.

\begin{clm}

All but $o(n)$ vertices in $B_i$ are connected to all the vertices in every $A_j$ with $j\neq i$.

\end{clm}

\begin{proof}
Consider a copy of $H$ in $G'$ containing a vertex $v\in B_i$. If a vertex of it is replaceable by $A_j$ with respect to $H$, then $v$ is connected to all the vertices in $A_j$. Otherwise every copy of $H_0$ obtained by removing a vertex of $H$ not in $B_i$ must belong to the $o(n^{|V(H_0)|})$ exceptional copies in Claim \ref{clam}. Observe that only vertices in one part of $H$ can belong to $B_i$, as vertices in other parts are connected to $c\ge a$ vertices of $A_i$.

For each vertex $v$,  $\Theta(n^{|V(H)|-1})$ copies of $H$ contain $v$, thus we obtain $\Omega(n^{|V(H)|-2})$ copies of $H_0$ this way, as a copy of $H_0$ might be obtained $O(n)$ ways. If we have $\Omega(n)$ exceptional vertices, they would belong to $\Omega(n^{|V(H)|-1})=\Omega(n^{|V(H_0)|})$ not nice copies of $H_0$, a contradiction with Claim \ref{clam}.
\end{proof}

\begin{clm}
For every $i$, every vertex of $B_i$ is connected to less than $a$ vertices of $B_i$.
\end{clm}

\begin{proof}
Recall that we started with picking an arbitrary $H'$. We obtained that $n-o(n)$ vertices of $G$ must be connected to every vertex of $k-1$ partite classes of that $H'$, let $Q$ be their set. Consider an arbitrary $u\in Q$, that belongs to, say $B_1$. Then we obtain another copy of $H'$ from the original one if we delete a vertex of $A_1$ and add $u$ instead. Let us denote this copy by $H''$. Applying the same for $H''$, we obtain that $n-o(n)$ vertices of $G$ are connected to every vertex of $k-1$ partite classes of $H''$. In particular, if $v\in B_j$ is connected to every vertex of $k-1$ partite classes of $H''$, the missing partite class has to be $A_j$ (which is a partite class of $H''$). Therefore, $v$ is connected to the first partite class of the new copy of $H'$, in particular to $u$. Thus every $u \in Q\cap B_1$ is connected to all but $o(n)$ vertices in every $B_j$ for $j>1$. 

If a vertex in $B_j\cap Q$ is connected to $a$ vertices in $B_j\cap Q$, then these $a+1$ vertices with $a$ vertices from classes of $H'$ form a copy of $F$, a contradiction. This implies that for every $j$, $|B_j\cap Q|=n/k+o(n)$. Indeed, the copies of $H$ inside $Q$ are all formed by taking a partite class from every $B_j$ by the above observation, thus their number is at most $y:=\prod_{j=1}^k \binom{|B_j\cap Q|}{b}\le \cN(H,T_k(n))$. It is easy to see that if the sizes of the sets $B_j$ are less balanced, then $y$ decreases by $\Omega(n^{|V(H)|})$. On the other hand, the number of copies of $H$ containing a vertex outside $Q$ is $o(n^{|V(H)|})$, thus the total number of copies of $H$ in $G$ is less than $\cN(H,T_k(n))$, a contradiction.

This also shows that for every copy of $H_0$ inside $Q$, its partite sets are contained in distinct $B_j$'s. Indeed, we can essentially repeat the argument from the first paragraph of the proof of this theorem. Again, we call an edge $uv$ an extra edge if $u,v\in B_j\cap Q$. If a copy of $H_0$ contains the extra edge $uv$, that means each other vertex of that $H_0$ is adjacent to at least one of $u$ and $v$, thus $H_0$ contains at most $2a-4$ other vertices from $B_j\cap Q$. Therefore, if a copy of $H_0$ inside $Q$ contains a set $U$ of more than $2a-2$ vertices from $B_j$, then there are no extra edges inside $U$, thus no edges at all inside $U$. Thus $U$ is a subset of a partite set of $H$, in particular $|U|\le b$. Therefore, the only way to choose $kb-1$ vertices from the sets $B_j\cap Q$ is to choose $b-1$ vertices from one of them and $b$ vertices from the other sets, and they have to form the partite sets of $H$.

Consider a vertex $v\in B_i\setminus Q$. Assume first that for some $j\neq i$, $v$ is connected to at most $\alpha n$ vertices of $B_j$ for some $\alpha<1/k$. Consider the copies of $H$ containing $v$. There are $o(n^{|V(H)|-1})$ copies of $H$ containing $v$ that also contain another vertex outside $Q$. Consider those copies of $H$ that have all the vertices in $Q$ (except for $v$), i.e. a copy of $H_0$ inside $Q$ that forms a copy of $H$ with $v$.
Then one of the partite classes of that $H_0$ is inside $B_j$, thus the vertices are chosen from the $\alpha n$ neighbors of $v$ in $B_j$. This shows that $v$ is in less than $x$ copies of $H$ altogether, a contradiction with Claim \ref{turi}.

Assume now that a vertex $v\in B_i$ is connected to $a$ vertices of $B_i$. Then for $j\neq i$, $v$ has $n/k-o(n)$ neighbors in $B_j$. We will build a copy of $F$. The one-element partite class is $v$. Its $a$ neighbors in $B_i$ form another partite class. Then we go through the sets $B_j\cap Q$ ($j\neq i$) one by one. We always have that the $a+1$ vertices we picked from $B_i$ have $n/k-o(n)$ neighbors in $B_j\cap Q$, and the already picked other vertices are connected to all of those, except for $o(n)$. Therefore, we can always pick $a$ vertices from $B_j\cap Q$ that are connected to all the vertices picked earlier. This way we obtain a copy of $F$, a contradiction.
\end{proof}
The above claim implies that in a copy of $H$, vertices of $B_i$ cannot belong to two different partite classes. Therefore, every $B_i$ contains a partite class. Let $G''$ be the graph we obtain by deleting the edges inside $B_i$ for every $i$. Then $\cN(H,G)=\cN(H,G'')\le \cN(H,T_k(n))$, where the inequality follows from the facts that $G''$ is $K_{k+1}$-free and $H$ is $(k+1)$-Tur\'an-good.
\end{proof}

Now we can prove Theorem \ref{critver}, which we restate here for convenience.

\begin{theorem*}
    There exists an $F$-Tur\'an-good graph if and only if $F$ has a color-critical vertex.
\end{theorem*}

\begin{proof}
Assume first that $F$ does not have a color-critical vertex and let $k=\chi(F)$. Let $T'_{k-1}(n)$ be obtained from $T_{k-1}(n)$ by taking a vertex $v$ from a largest partite set of $T_{k-1}(n)$, and connect it to every other vertex. Then deleting $v$ from $T'_{k-1}(n)$ we obtain a $(k-1)$-partite graph. As deleting any vertex from $F$ we obtain a $k$-chromatic graph, $T'_{k-1}(n)$ is $F$-free. 

We claim that for any graph $H$, for $n$ large enough, either $\cN(H,T'_{k-1}(n))>\cN(H,T_{k-1}(n))$ or $\cN(H,T'_{k-1}(n))=\cN(H,T_{k-1}(n))=0$. Indeed, if there is an $H$ in $T_{k-1}(n)$, then there is one avoiding $v$, but using a vertex $u$ from the same partite set, and a vertex $w$ from another partite set. Then we can replace $w$ with $v$ to obtain a copy of $H$ that is in $T'_{k-1}(n)$, but not in $T_{k-1}(n)$.

If $\cN(H,T_{k-1}(n))=0$, then the Tur\'an graph may be extremal if $H$ contains $F$, but then $H$ is not $F$-Tur\'an-good. If $\cN(H,T'_{k-1}(n))>\cN(H,T_{k-1}(n))$, then the Tur\'an graph is not extremal, finishing the proof.

Assume now that $F$ has a color-critical vertex. Then $F$ is a subgraph of a complete $k$-partite graph $K_{1,a,\dots,a}$, thus Theorem \ref{compl} finishes the proof.
\end{proof}

\section{Concluding remarks}

$\bullet$ We showed that if a $k$-chromatic graph $F$ has a color-critical edge and $\ex(n,K_k,F)=o(n^{k-1})$, then several $k$-Tur\'an-good graphs are also $F$-Tur\'an-good. The proof deals only with those copies of $K_k$ that are in the vertex set of a copy of $H$. This suggests to study a local variant of generalized Tur\'an problems. We say that a subgraph $G'$ of $G$ is $F$-free with respect to $G$ if there is no copy of $F$ induced on $V(G')$. How many subgraph of an $n$-vertex graph $G$ can be isomorphic to $H$ and be $F$-free with respect to $G$ at the same time?

If $F=K_k$ and the largest number as an answer to the above question is obtained when $G=T_{k-1}(n)$, then it is immediate that $H$ is $k$-Tur\'an-good. The argument used in the proof of Theorem \ref{turgoodcrit} shows that for any other $F$ with a color-critical edge and $\ex(n,K_k,F)=o(n^{k-1})$, we have that $H$ is also $F$-Tur\'an-good.

$\bullet$ A theme of this paper is to extend some results on $k$-Tur\'an good graphs to $F$-Tur\'an good graphs when $F$ is a $k$-chromatic graph with a color critical edge. This motivates the question: is every $k$-Tur\'an-good graph also $F$-Tur\'an-good?

$\bullet$ Maybe even more is true than what is suggested in the previous paragraph. Let us call a graph $H$ weakly $F$-Tur\'an-good if the number of copies of $H$ is maximized by complete multipartite graph among $F$-free graphs on $n$ vertices, provided $n$ is large enough. Is every weakly $K_k$-Tur\'an-good graph also weakly $F$-Tur\'an-good?

\end{document}